\documentclass[12pt]{amsart}
\usepackage{txfonts}
\usepackage{amscd,amssymb,mathabx,subfigure}
\usepackage[all]{xy}
\usepackage{graphicx,calrsfs}
\usepackage[colorlinks,plainpages,backref,urlcolor=blue]{hyperref}
\usepackage{tikz}
\usetikzlibrary{calc}
\usepackage{mdwlist}
\usepackage{verbatim}
\usepackage{color}

\newtheorem{theorem}{Theorem}[section]

\newtheorem{lemma}[theorem]{Lemma}
\newtheorem{prop}[theorem]{Proposition}

\theoremstyle{definition}
\newtheorem{definition}[theorem]{Definition}

\newtheorem{remark}[theorem]{Remark}
\newtheorem{conjecture}[theorem]{Conjecture}

\newcommand{\Z}{\mathbb{Z}}

\newcommand{\R}{\mathbb{R}}
\newcommand{\C}{\mathbb{C}}

\DeclareMathAlphabet{\pazocal}{OMS}{zplm}{m}{n}
\newcommand{\A}{{\pazocal{A}}}
\newcommand{\B}{{\pazocal{B}}}

\def\dot{\mathchar"013A}
\newcommand{\hdot}{{\raise1pt\hbox to0.35em{\Huge $\dot$}}}

\definecolor{dkgreen}{RGB}{0,100,0}
\definecolor{dkbrown}{RGB}{139,69,19}

\begin{document}
\date{April 24, 2018}

\title[Freeness and near freeness for small degree line arrangements]%
{Freeness and near freeness are combinatorial for line arrangements in small degrees}

\author[A. Dimca]{Alexandru Dimca}
\address{Universit\'e C\^ ote d'Azur, CNRS, LJAD, France  }
\email{dimca@unice.fr}

\author[D. Ibadula]{Denis~Ibadula}
\address{Ovidius University, Faculty of Mathematics and Informatics, 124 Mamaia Blvd., 900527 Constan\c{t}a, Romania}
\email{denis.ibadula@univ-ovidius.ro}

\author[A. M\u acinic]{Anca~M\u acinic}
\address{Simion Stoilow Institute of Mathematics, 
 Bucharest, Romania}
\email{Anca.Macinic@imar.ro}

\subjclass[2010]{Primary 32S22}

\keywords{line arrangement; free arrangement;  nearly free arrangement; Terao's conjecture; intersection lattice}

\begin{abstract} 
We prove Terao conjecture saying that the freeness is determined by the combinatorics for arrangements of 13 lines in the complex projective plane and that the property of being nearly free is combinatorial for line arrangements of up to 12 lines in the complex projective plane.
\end{abstract}
 
\maketitle

\section{Introduction} 

Let $\A$ be an arrangement of $d$ lines in the complex projective plane (or, equivalently, a central arrangement of planes in $\C^3$), defined by the equation $f=0$, with $f \in S:=\C[x,y,z]$ a polynomial of degree $d$. The minimal degree of a Jacobian relation for $f$ is the integer $mdr(f)$, defined to be the smallest integer $m \geq 0$  such that there is a nontrivial relation
$$\rho(f): a f_x + b f_y + c f_z = 0$$
among the partial derivatives $f_x, \;f_y, \; f_z$ of $f$ with coefficients $a, b, c$ homogeneous polynomials of degree $m$.
Let $AR(f)$ be the graded $S$-module  of all Jacobian relations of $f$ as above. The arrangement $\A$ is called {\it free} when the $AR(f)$ is free  as an $S$-module. The exponents of the free arrangement $\A$ are defined as the degrees of the elements of a basis for $AR(f)$. Notice that $AR(f)$ is isomorphic to the derivation module $D(\A) =\{ \theta \in Der S  \; |\; \theta(f) = 0\}$, so this definition coincides to the one in \cite{OT}.

To an arrangement of hyperplanes one associates a geometric lattice, the lattice of intersection of various subsets of the set of hyperplanes of $\A$, ordered by reverse inclusion, denoted $L(\A)$. A property of an arrangement of hyperplanes $\A$ is called {\it combinatorial} if it depends only on the lattice  isomorphism class of the lattice $L(\A)$. Important open questions regard the combinatoriality of certain properties or invariants associated to hyperplane  arrangements. Among them,  { \it Terao conjecture}, which is the subject of intense research in the field (see for instance \cite{A2, D1, DHA,Y0,Y}), occupies a central place.

\begin{conjecture}[Terao]
\label{terao}
The property of an arrangement of being free is combinatorial.
\end{conjecture}

We prove in this note  that Terao conjecture holds for line arrangements having $13$ lines and make a step towards proving the conjecture for arrangements of $14$ lines. It is known that for arrangements of cardinal at most $12$, the conjecture holds (see \cite{FV}, \cite{ACKN}). It is natural to look as the next step to arrangements of $13$ lines. 
Moreover, this pursuit is justified by the fact that it is known that $13$ is the smallest cardinal for a line arrangement to be free, but not recursively free (\cite{ACKN}). 

Another result concerns the combinatorial nature of the property of an arrangement of being nearly free. We give a positive answer for arrangements of at most $12$ lines.

The proof of our results relies in fact on the interplay between free and nearly free properties of line arrangements. Some necessary definitions and results are recalled in the next section. In the third section we prove Theorem \ref{thm:main}, the Terao conjecture for arrangements of $13$ lines, and in the fourth section we prove that near freeness is combinatorial for arrangements of at most $12$ lines (Theorem \ref{thm:nfree}) and finally that the Terao conjecture for arrangements of $14$ lines can be reduced to the problem of combinatoriality of near freeness for arrangements of $13$ lines, in Proposition  \ref{prop:14free}.\\

It has recently been brought to our attention by Hiraku Kawanoue that both he
and Torsten Hoge have also confirmed the Terao conjecture in the case $|\A| = 13$.
Their approaches seem to be through computer aided computations, and have not been published, not even in a preprint form. 
\section{The results we need}

Yoshinaga has introduced in \cite{Y0,Y} a technique of study of freeness of arrangements through multiarrangements. A {\it multiarrangement} is simply a pairing of an arrangement $\A$ with a map $m: \A \rightarrow \Z_{\geq 0}$, called {\it multiplicity}. An arrangement can then be identified with a multiarrangement with constant multiplicity $m \equiv 1$. A notion of freeness (and a corresponding notion of exponents) for multiarrangements is defined (\cite[Def. 1.12]{Y}). It is easy to see that an arrangement in $\C^2$ is always free with exponents $(1, |\A| - 1)$. Similarly, it is true that a multiarrangement in $\C^2$ is free and its exponents $(d_1, d_2)$ satisfy $d_1+d_2 = |\A|$, however these exponents are not trivial to compute. Actually, as we recall below using \cite{Y},  their computation is related to the freeness property for arrangements in $\C^3$.

\begin{prop}
\label{prop:Y1.23} \cite[Prop.1.23]{Y}
Let $(\A, m)$ be a $2-$multiarrangement. We may assume that $m_i = m(H_i)$ satisfies $m_1 \geq m_2 \geq \dots \geq m_n >0$. Set $m = \sum_{i=1}^n m_i$.
\begin{enumerate}
\item If $m_1 \geq \frac{m}{2}$, then the exponents are $exp(\A, m) = (m_1, m - m_1)$
\item If $n \geq \frac{m}{2}+1$, then  $exp(\A, m) = (m-n+1, n-1)$
\item If $m_1 = m_2 = \dots = m_n = 2$, then $exp(\A, m) = (n, n)$
\end{enumerate}
\end{prop} 

To an arbitrary arrangement $\A$ one may associate  certain multiarrangements as restrictions. For a hyperplane $H \in \A$, consider the arrangement induced on $H, \;\A^H$. Define a multiplicity map $m^H$ on $\A^H$ by 

\begin{equation}
\label{formula:m^H}
X \in \A^H \mapsto \#\{K \in \A| \; X \subset K\} - 1
\end{equation}

We call the multiarrangement $(\A^H, m^H)$ the Ziegler multirestriction of $\A$ onto $H$. 
The next result gives a set of sufficient conditions for the combinatoriality of the freeness property in terms of Ziegler restrictions.

\begin{prop}\cite[Prop 1.47]{Y}
\label{prop:Y1.47} 
Let $\A$ be a projective line arrangement such that there exist a hyperplane $H \in \A$ with the Ziegler multirestriction $(\A^H, m^H)$ satisfying one of the conditions in Proposition \ref{prop:Y1.23}. Then the freeness of $\A$ implies the freeness of any other arrrangement in the same lattice isomorphism class.
\end{prop}

A recent notion, of {\it near freeness}, for plane projective curves was introduced in \cite{DS} by Dimca-Sticlaru. The authors conjecture that any rational cuspidal curve that is not a free divisor is nearly free.
We will consider here this notion only in the restricted context of projective line arrangements. 

Let $J_f$ be the Jacobian ideal of $f$, that is, the ideal spanned by the partial derivatives of $f$, and denote by $\hat{J_f}$ the saturation of $J_f$ with respect to the maximal ideal $m = (x, y, z)$ in $S$. 

\begin{definition}(\cite{DS})
\label{def:nfree}
An arrangement is called { \it nearly free} if the quotient graded $S$-module $N(f) = \hat{J_f}/J_f$ is nontrivial and $\dim N(f)_k \leq 1$, for any $k$.
\end{definition}

Moreover, a parallel  notion of (near)exponents is introduced in \cite{DS}.
A result mirroring the decomposition of the characteristic polynomial with respect to the exponents for free arrangements is then stated for nearly free ones. To state it, we recall the following. The projective arrangement $\A$ can be naturally identified with a central arrangement $\overline{\A}$ of planes in $\C^3$. We define the characteristic polynomial of $\A, \;\chi(\A; t)$, by relation to the characteristic polynomial of $\overline{\A}$: since $\overline{\A}$ is central,  $\chi(\overline{\A};t)$ always has as a factor $t-1$. 
Let us define then  $\chi(\A; t) = \chi(\overline{\A}; t)/(t-1)$.

\begin{prop}\cite[Prop 3.12]{DS}
\label{prop:DS} Let $\A$ be nearly free with $nexp = (d_1, d_2)$. Then $d_1+d_2 = |\A|$ and 
$$
\chi(\A; t) = (t-d_1)(t-d_2+1)+1
$$
\end{prop}

The notions of free and nearly free arrangements are subtly connected, as a series of results in \cite{AD} linking freeness and near-freeness through deletion-restriction type statements show (\cite[Thm 5.11, Thm. 5.7, Thm. 5.10]{AD}). We will use in our proof the following one.

\begin{theorem}(\cite[Thm 5.11]{AD})
\label{thm:5.11_del_restr}
Let $\A$ be an arrangement of lines in the complex projective plane, $H \in \A$ and $\B = \A \setminus \{H\}$. Also, let $d_1 \leq d_2$ be two non-negative integers. Then any two of the following imply the third:
\begin{enumerate}
\item $\A$ is free with $exp(\A) = (d_1, d_2)$
\item $\B$ is nearly free with $nexp(\B) = (d_1, d_2)$
\item $|\A^H| = d_1$
\end{enumerate}
\end{theorem}

Finally, it is worth noticing that the connection between the two notions is underlined by the fact that a series of results that hold for free arrangements seem to admit a near free counterpart (see for instance the result below, compared to  \cite[Thm 1.1(3)]{A}). 

\begin{theorem}\cite[Thm 5.8]{AD}
\label{thm:thm5.8_AD}
Let $\A$ be an arrangement of lines in the complex projective plane with $\chi(\A, t) = t^2 - b_1t + b_2$, where $b_1 = |\A| - 1$. Let $\chi(\A, t) = (t-a)(t-b)+1$ with $a, b \in \R, \; a \leq b, \; a+b = b_1$. Then $\A$ is nearly free if there is $H \in \A$ such that 
\begin{enumerate}
\item $|\A^H| = b+1$ or,
\item $|\A^H| = a+1$ and $b \neq a+2$
\end{enumerate}
\end{theorem}

Lastly, we need to recall some combinatoric ingredients.
We will denote by $n_k$ the number of points of multiplicity $k$ of $\A$. Some restrictions apply to these multiplicities, for instance the easily deducible equality:
	
\begin{equation}
\label{eq:suma_comb}
{d \choose 2} = \Sigma_{k=2}^{d} n_k{k \choose 2}
\end{equation}
A highly non-trivial restriction on the multiplicities is given by the Hirzebruch inequality (provided that $n_d=n_{d-1}=0$, see \cite{H}):

\begin{equation}
\label{eq:hirzebruch}
n_2 +\frac{3}{4}n_3 \geq d+ \Sigma_{k \geq 5}(k-4)n_k
\end{equation}

\section{Terao conjecture for $13$ lines arrangements}

\begin{theorem}
\label{thm:main}
 Terao conjecture is true for arrangements of 13 lines in the complex projective plane.
\end{theorem}
For a line arrangement $\A$, we denote by $m(\A)$ the maximal multiplicity of the intersection points in $\A$.
To prove Terao's Conjecture in the case $d=13$, it is enough 
to only consider the case

\bigskip

($*$) $d_1=6$, $m(\A) \in \{4,5\}$ and any line in $\A$ contains at most 6 intersection points.

\bigskip

\noindent 
Indeed, the case $d_1 \leq 5$ follows from \cite[Corollary 5.5]{A}, and hence  the exponents of $\A$ can be assumed to be $d_1=d_2=6$. Assume that there is a line $L \in \A$ containing at least 7 points. Then \cite[Theorem 2.7]{ACKN} implies that in these conditions freeness is determined by the combinatorics.
Hence the Terao's conjecture holds for $\A$. When $m(\A) \leq 3$ we use \cite[Proposition 1.3]{D2} to see that there are no free arrangements in this case. And for $m(\A) \geq d_1=6$ we apply \cite[Corollary 1.4]{D2}.
 
From now on, unless otherwise stated, $\A$ is a 13 lines arrangement in the complex projective plane that has only multiple points of multiplicity up to $5$ and minimal degree relations $mdr =6$.  When $\A$ is free, this amounts to $d_1 = mdr = 6$.
We will prove the combinatorial nature of the freeness property in this setting.

\begin{prop}
\label{n_5=3}
 $n_5 \leq 3$ for arrangements of $13$ lines. 
\end{prop}

\begin{proof}
We will call two points {\it collinear} if they are situated on a line in $\A$. \\

\noindent Assume $n_5 \geq 3$.  Then obviously there are at least two collinear quintuple points in the arrangement. Moreover, a third quintuple point should be collinear to at least one of the previous two collinear quintuple points (so, in any case, one of the configurations (a), (b) from Figure \ref{fig:3quintuple} happens, as subarrangements). If all three are situated on the same line (Figure \ref{fig:3quintuple}  (a)), then immediately $n_5=3$.\\
 Otherwise, assume no three quintuple points are collinear. If there are three quintuple points in the arrangement such that each pair of two are collinear (Figure \ref{fig:3quintuple} (c)), then it is immediate that $n_5=3$. So, assume none of the previous two situations (Figure \ref{fig:3quintuple} (a) or  (c)) occur. We have then a pencil of five lines with base point $Q_0$, and two of the lines in the pencil, $d_1, \;  d_2$ each must contain an additional quintuple point, $Q_1, \; Q_2$, such that  $Q_1, \; Q_2$ are non-collinear in $\A$. $Q_2$ contains the four lines that are not part of the quintuple points on the line $d_1$ (see Figure \ref{fig:3quintuple} (d)). The existence of an additional quintuple point $Q_3$ is then not possible.
\end{proof}

\begin{figure}[h]
\label{fig:3quintuple}
\begin{tikzpicture}[scale=0.75]
\draw[style=thick] (-12,1) -- (-4,1);
\node at (-10,1){$\bullet$};
\draw[style=thick, rotate around={-36: (-10,1)}] (-11,1) -- (-9,1);
\draw[style=thick, rotate around={-72: (-10,1)}] (-11,1) -- (-9,1);
\draw[style=thick, rotate around={-108: (-10,1)}] (-11,1) -- (-9,1);
\draw[style=thick, rotate around={-144: (-10,1)}] (-11,1) -- (-9,1);
\node at (-8,1){$\bullet$};
\draw[style=thick, rotate around={-36: (-8,1)}] (-9,1) -- (-7,1);
\draw[style=thick, rotate around={-72: (-8,1)}] (-9,1) -- (-7,1);
\draw[style=thick, rotate around={-108: (-8,1)}] (-9,1) -- (-7,1);
\draw[style=thick, rotate around={-144: (-8,1)}] (-9,1) -- (-7,1);
\node at (-6,1){$\bullet$};
\draw[style=thick, rotate around={-36: (-6,1)}] (-7,1) -- (-5,1);
\draw[style=thick, rotate around={-72: (-6,1)}] (-7,1) -- (-5,1);
\draw[style=thick, rotate around={-108: (-6,1)}] (-7,1) -- (-5,1);
\draw[style=thick, rotate around={-144: (-6,1)}] (-7,1) -- (-5,1);
\node at (-9,-1.5) {(a)};
\draw[style=thick] (0,1) -- (6,1);
\node at (2,1){$\bullet$};
\draw[style=thick, rotate around={-36: (2,1)}] (1,1) -- (3,1);
\draw[style=thick, rotate around={-72: (2,1)}] (1,1) -- (3,1);
\draw[style=thick, rotate around={-108: (2,1)}] (1,1) -- (5,1);
\draw[style=thick, rotate around={-144: (2,1)}] (1,1) -- (3,1);
\node at (4,1){$\bullet$};
\draw[style=thick, rotate around={-36: (4,1)}] (3,1) -- (5,1);
\draw[style=thick, rotate around={-72: (4,1)}] (3,1) -- (5,1);
\draw[style=thick, rotate around={-108: (4,1)}] (3,1) -- (5,1);
\draw[style=thick, rotate around={-144: (4,1)}] (3,1) -- (5,1);
\node at (1.35,-1) {$\bullet$};
\draw[style=thick] (0.35,-1) -- (2.35,-1);
\draw[style=thick, rotate around={-36: (1.35,-1)}] (0.35,-1) -- (2.35,-1);
\draw[style=thick, rotate around={-72: (1.35,-1)}] (0.35,-1) -- (2.35,-1);
\draw[style=thick, rotate around={-144: (1.35,-1)}] (0.35,-1) -- (2.35,-1);
\node at (4,-1.5) {(b)};

\draw[style=thick] (-12,-7.5) -- (-6,-7.5);
\node at (-11,-7.5){$\bullet$};
\node at (-7,-7.5){$\bullet$};
\node at (-9,{-7.5+2*sqrt(3)}){$\bullet$};
\draw[style=thick, rotate around={60: (-11,-7.5)}] (-12,-7.5) -- (-6,-7.5);
\draw[style=thick, dashed, rotate around={15: (-11,-7.5)}] (-12,-7.5) -- (-10,-7.5);
\draw[style=thick, dashed, rotate around={30: (-11,-7.5)}] (-12,-7.5) -- (-10,-7.5);
\draw[style=thick, dashed, rotate around={45: (-11,-7.5)}] (-12,-7.5) -- (-10,-7.5);
\draw[style=thick, rotate around={-60: (-7,-7.5)}] (-12,-7.5) -- (-6,-7.5);
\draw[style=thick, dashed, rotate around={-15: (-7,-7.5)}] (-8,-7.5) -- (-6,-7.5);
\draw[style=thick, dashed, rotate around={-30: (-7,-7.5)}] (-8,-7.5) -- (-6,-7.5);
\draw[style=thick, dashed, rotate around={-45: (-7,-7.5)}] (-8,-7.5) -- (-6,-7.5);
\draw[style=thick, dashed] (-9,{-6.5+2*sqrt(3)}) -- (-9,{-8.5+2*sqrt(3)});
\draw[style=thick, dashed, rotate around={-15: (-9,{-7.5+2*sqrt(3)})}] (-9,{-6.5+2*sqrt(3)}) -- (-9,{-8.5+2*sqrt(3)});
\draw[style=thick, dashed, rotate around={15: (-9,{-7.5+2*sqrt(3)})}] (-9,{-6.5+2*sqrt(3)}) -- (-9,{-8.5+2*sqrt(3)});
\node at (-9,-8.5) {(c)};

\draw[style=thick] (0,-5) -- (7,-5);
\node at (0.2,-4.7){$d_1$};
\node at (2.2,-3.4){$d_2$};
\node at (2,-5){$\bullet$};
\node at (3.2,-4.5){$Q_0$};
\draw[style=thick, rotate around={-36: (2,-5)}] (1,-5) -- (3,-5);
\draw[style=thick, rotate around={-72: (2,-5)}] (1,-5) -- (3,-5);
\draw[style=thick, rotate around={-108: (2,-5)}] (0,-5) -- (6,-5);
\draw[style=thick, rotate around={-144: (2,-5)}] (1,-5) -- (3,-5);
\node at (5.5,-5){$\bullet$};
\node at (6.9,-4.5) {$Q_2$};
\draw[style=thick, rotate around={-36: (5.5,-5)}] (4.5,-5) -- (6.5,-5);
\draw[style=thick, rotate around={-72: (5.5,-5)}] (4.5,-5) -- (6.5,-5);
\draw[style=thick, rotate around={-108: (5.5,-5)}] (4.5,-5) -- (6.5,-5);
\draw[style=thick, rotate around={-144: (5.5,-5)}] (4.5,-5) -- (6.5,-5);
\node at (1.35,-7) {$\bullet$};
\node at (2.8,-6.8) {$Q_1$};
\draw[style=thick, dashed] (0.35,-7) -- (2.35,-7);
\draw[style=thick, dashed, rotate around={-36: (1.35,-7)}] (0.35,-7) -- (2.35,-7);
\draw[style=thick, dashed, rotate around={-72: (1.35,-7)}] (0.35,-7) -- (2.35,-7);
\draw[style=thick, dashed, rotate around={-144: (1.35,-7)}] (0.35,-7) -- (2.35,-7);
\node at (4,-8.5) {(d)};

\end{tikzpicture}
\caption{$3$ quintuple points configurations}
\end{figure}
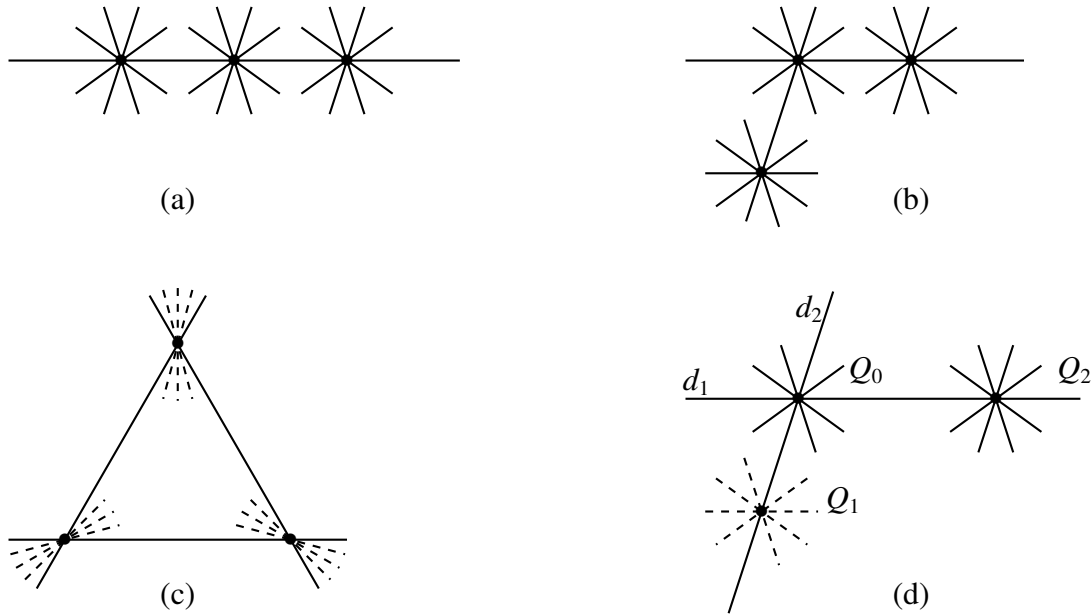

\begin{remark}
\label{remark:restrictions}
Let $\A$ be a free  arrangement of $13$ lines. 
\begin{enumerate}
\item[(i)] 
One can assume that all the lines in $\A$ contain at least $5$ and at most $6$ multiple points:  if the arrangement has a line with at least $7$ multiple points then, by \cite[Propositions 1.47, 1.23(ii)]{Y}, the freeness depends only on the lattice isomorphism class; on the other hand,  if the arrangement has a line with at most $4$ multiple points, since its characteristic polynomial is a perfect square, any other arrangement $\A'$ in the lattice isomorphism class of $\A$ is either free or nearly free, by \cite[Theorem 1.5]{AD}. Let $\A'$ be such a nearly free arrangement. Since $\A$ and $\A'$ are lattice isomorphic, they have the same Tjurina number, hence the only possibility for $\A'$ is to have exponents $(5,8)$. Were the multiplicity of $\A'$ equal to $5$, \cite[Cor.1.7]{D2} would imply $\A$ is also nearly free, contradiction. In conclusion, when $m(\A) = 5$, we can assume that $\A$ does not have lines  with at most $4$ multiple points.

\item[(ii)] 
By \cite[Propositions 1.47, 1.23(iii)]{Y} for arrangements with a line that only contains triple points (hence $6$ triple points) the freeness depends only on the lattice isomorphism class. 
\item[(iii)] 
By $(i)$ and $(ii)$, to test the Terao conjecture for arrangements of $13$ lines one only needs to look at arrangements $\A$ for which each line $ H \in \A$ is in one of the following situations, where $n^H_i$ is the number of multiple points of multiplicity $i$ on $H$:
\begin{enumerate}
\item[($a_0$)] $n^H_5 = 0, \;n^H_4 = 4, \;  n^H_3 = 0, \; n^H_2 = 0$
\item[(a)] $n^H_5 = 0, \;n^H_4 = 1, \;  n^H_3 = 4, \; n^H_2 = 1$
\item[(b)] $n^H_5 = 0, \;n^H_4 = 2, \;  n^H_3 = 3, \; n^H_2 = 0$
\item[(c)] $n^H_5 = 0, \;n^H_4 = 2, \;  n^H_3 = 2, \; n^H_2 = 2$
\item[(d)] $n^H_5 = 0, \;n^H_4 = 3, \;  n^H_3 = 1, \; n^H_2 = 1$
\item[(e)] $n^H_5 = 0, \;n^H_4 = 3, \;  n^H_3 = 0, \; n^H_2 = 3$ 
\item[(f)]$n^H_5 = 1, \;n^H_4 = 0, \;  n^H_3 = 3, \; n^H_2 = 2$ 
\item[(g)] $n^H_5 = 1, \;n^H_4 = 0, \;  n^H_3 = 4, \; n^H_2 = 0$ 
\item[(h)] $n^H_5 = 1, \;n^H_4 = 1, \;  n^H_3 = 1, \; n^H_2 = 3$ 
\item[(i)] $n^H_5 = 1, \;n^H_4 = 1, \;  n^H_3 = 2, \; n^H_2 = 1$
\item[(j)] $n^H_5 = 1, \;n^H_4 = 2, \;  n^H_3 = 0, \; n^H_2 = 2$ 
\item[(k)] $n^H_5 = 2, \;n^H_4 = 0, \;  n^H_3 = 0, \; n^H_2 = 4$
\item[(l)] $n^H_5 = 2, \;n^H_4 = 0, \;  n^H_3 = 1, \; n^H_2 = 2$ 
\end{enumerate}

Notice that $\sum_i n_i^H = |\A^H|$ and in fact $n_i^H= |(m^H)^{-1}(i-1)|$, that is,  the cardinal of the pre-image of $i-1$ through the multiplication map $m^H$ defined by \eqref{formula:m^H}.
\end{enumerate}
\end{remark}

To prove Theorem \ref{thm:main}, it is enough to see that the following property holds.
\begin{lemma}
\label{lemma:2lines_6points}
For $\A$ as above, there exist two lines in $\A$, each containing $6$ multiple points, such that the two lines do not intersect in a double point. 
\end{lemma}
\begin{proof}
Recall that $\A$ is as described by Remark \ref{remark:restrictions}, that is, a free $13$ lines arrangement such that each line is of one of the types ($a_0$) through (l).
 Denote the number of lines of type ($a_0$), (a), (b), (c), (d), (e), (f), (g), (h), (i), (j), (k), (l) in $\A$ by 
 $a_0, a, b, c, d, e, f, g, h, i, j, k, l$.
  Since $\A$ is free, it follows from \cite[Cor 1.2]{D1} that $n_2+4n_3+9n_4+16n_5=108$ and by \eqref{eq:suma_comb} that  $n_2+3n_3+6n_4+10n_5=78$.
  We make a discussion on the number of quintuple points of the arrangement, with the purpose of counting the number of lines in $\A$ having $6$ multiple points. To do that, we solve in each case a system of linear equations (including the two above) with $n_i, \; 2 \leq i \leq 4,$ and  $a_0, a, b, c, d, e, f, g, h, i, j, k, l$ as variables. The third, fourth and fifth equations of the system are a count for the number of double, triple, respectively quadruple points of the arrangement. \\
  
    A quick look at the list in Remark \ref{remark:restrictions}(iii) shows that, if $n_5 \leq 1$ and $\A$ has at least $5$ lines  having $6$ multiple points then by removing one of those lines one is left with an arrangement containing at least another line with $6$ multiple points. When $n_5 \geq 2$, it is enough for $\A$ to contain at least $6$  lines with $6$ multiple points to reach the same conclusion.

 {\it Case $n_5 = 0$} Solving the linear system \\

\noindent $|\; n_2+4n_3+9n_4=108 \\
|\; n_2+3n_3+6n_4=78 \\
|\; -2n_2+3e+d+2c+a=0 \\
|\; -3n_3+d+3b+2c+4a=0 \\
|\; -4n_4+4a_0+a+2c+2b+3d+3e=0 \\
|\; a_0+a+b+c+d+e=13 
$\\

\noindent shows that the number of lines of the arrangement with at most $5$ multiple points (that is, the sum $b+d+a_0$) is equal to $12-n_4-a_0$. If $a_0 >0$, this implies $n_4 \geq 4$, hence $b+d+a_0 \leq 7$, so there are at least $6$ lines in the arrangement having $6$ multiple points. If $a_0 = 0$ we get $b+d = 8- \frac{n_2}{3}$, so there are at least $5$ lines in the arrangement having $6$ multiple points.

Recall that, by Remark \ref{remark:restrictions}(i) when $m(A) = 5$, then there are no lines in $\A$ of type $(a_0)$.

{\it Case $n_5 = 1$} Consider the solution for the linear system:\\

\noindent$|\; n_2+4n_3+9n_4=92\\
|\; n_2+3n_3+6n_4=68\\
|\; -2n_2+3e+d+2c+a+2f+3h+2j+i=0\\
|\; -3n_3+d+3b+2c+4a+3f+4g+h+2i=0\\
|\; -4n_4+a+2c+2b+3d+3e+h+2j+i=0\\
|\; a+b+c+d+e+f+g+h+i+j=13\\
|\; f+g+h+i+j=5
$\\

 The number of lines of the arrangement with $6$ multiple points (that is, the sum $a+c+e+f+h$) is equal to $4+n_4$. Moreover, $n_2 \geq 0$ implies $n_4 \geq 2$, hence  there are at least $6$ lines in the arrangement having $6$ multiple points.
 
 {\it Case $n_5 = 2$} We have to solve the linear system:\\

\noindent $|\;  n_2+4n_3+9n_4=76\\
|\; n_2+3n_3+6n_4=58\\
|\; -2n_2+3e+d+2c+a+2f+3h+2j+4k+2l+i=0\\
|\; -3n_3+d+3b+2c+4a+3f+4g+h+l+2i=0\\
|\; -4n_4+a+2c+2b+3d+3e+h+2j+i=0\\\
|\; a+b+c+d+e+f+g+h+i+j+k+l=13\\
|\; f+g+h+i+j+2k+2l= 10
$\\
 
One obtains that the total number of lines having $5$ multiple points, the sum $b+d+g+i+j+l$,  equals $6-n_4$, so the number of lines having $6$ multiple points must be at least $7$.

{\it Case $n_5 = 3$} We have to solve the linear system:\\

\noindent $|\; n_2+4n_3+9n_4=60\\
|\; n_2+3n_3+6n_4=48\\
|\; -2n_2+3e+d+2c+a+2f+3h+2j+4k+2l+i=0\\
|\; -3n_3+d+3b+2c+4a+3f+4g+h+l+2i=0\\
|\; -4n_4+a+2c+2b+3d+3e+h+2j+i=0\\
|\; a+b+c+d+e+f+g+h+i+j+k+l=13\\
|\; f+g+h+i+j+2k+2l= 15
$\\

One obtains that the total number of lines having $5$ multiple points, the sum $d+b+g+j+l+i$,  equals $3-n_4$, so the number of lines having $6$ multiple points must be at least $10$.

\end{proof}

\noindent {\it Proof  of Theorem \ref{thm:main}}\\

To prove Theorem \ref{thm:main}, we can assume that $\A$ is a free arrangement as described by Remark 
\ref{remark:restrictions}: 

\begin{enumerate}

\item $\A$ is a 13 lines arrangement in the complex projective plane.

\item  $\A$ has only multiple points of multiplicity up to $5$ and minimal degree relations $d_1=mdr =6$; this implies that the characteristic polynomial of $\A$ is $\chi_{\A}(t)=(t-6)^2$.

\item The lines in $\A$ contain either $5$ or $6$ multiple points (all lines are of type (a) to (l) as listed above), except when $m(\A) = 4$, in which case $\A$ may also contain lines of type $(a_0)$. 

\end{enumerate}

The above lemma immediately implies that $\A$ contains a line with $6$ multiple points such that if one removes that line from the arrangement the resulting $12$ lines arrangement still contains a line with $6$ multiple points.\\

In this hypothesis, one may apply Theorem \ref{thm:5.11_del_restr} to $\A$ and obtain a $12$ lines nearly free subarrangement $\B = \A \setminus \{H\}$ ( $H$ line with $6$ multiple points in $\A$) with nexp $(6, 6)$ such that  $\B$ still contains a line with $6$ multiple points. Consider another arrangement $\A'$ in the lattice isomorphism class of $\A$. We need to show that $\A'$ is also free. \\

Let $\B'$ be the subarrangement of $\A'$ obtained by removing the line $H'$ corresponding to $H$ through the lattice isomorphism, that is $\B$ and $\B'$ are lattice isomorphic. Then $\B'$ has characteristic polynomial $(t-5)(t-6)+1$ (Proposition \ref{prop:DS}). We apply Theorem \ref{thm:thm5.8_AD} to $\B'$, which contains a line with $6$ multiple points, to deduce that $\B'$ is also nearly free. By \cite[Cor 3.5]{DS}, two nearly free curves with the same degree and the same global Tjurina number have the same exponents, so $\B'$ has the same exponents as $\B$, that is $\B' = (6, 6)$.
Now again by Theorem \ref{thm:5.11_del_restr}, it follows that $\A'$ is free.

\section{Combinatorial freeness for $14$ lines arrangements via Combinatorial near freeness for up to $13$ lines arrangements}

The previous approach proves very fruitful. By a similar line of reasoning, this free/nearly free interplay can be used to show that near freeness is combinatorial, at least for arrangements of up to a certain cardinal.

\begin{prop} 
\label{prop:sufficient_cond}
Let $\A$ be a nearly free arrangement of at most $14$ lines with exponents $d_1 \leq  d_2$. 
 If there is a line in $\A$ having precisely $(d_2+1)$ multiple points, then all arrangements lattice isomorphic to $\A$ are also nearly free.
\end{prop}

\begin{proof}
The proof is immediate, and it relies on the combinatoriality of the freeness property of arrangements of up to $13$ lines (Theorem \ref{thm:main}) and \cite[Thm 5.10]{AD}. Take a nearly free arrangement $\A$ and another arrangement $\A'$ lattice isomorphic to $\A$. We need to prove that $\A'$ is also nearly free.
Let $H \in \A$ be such that $|\A^H| = d_2+1$ and $H' \in \A'$ the line corresponding to $H$ through the lattice isomorphism. Then $|\A'^{H'}| = d_2+1$ also. Since $\B = \A \setminus \{H\}$ is free by  \cite[Thm 5.10]{AD} and it is lattice isomorphic to  $\B' = \A' \setminus \{H'\}$, then $\B'$ is also free, by \ref{thm:main}. \cite[Thm 5.10]{AD} now completes the proof, showing that $\A'$ must be nearly free as well.
\end{proof}

It is also worth stating the following easy consequence of \cite[Thm 5.8]{AD}.

\begin{remark} 
\label{rem:sufficient_cond}
Let $\A$ be a nearly free arrangement with exponents $d_1 \leq  d_2$.
If one of the below properties holds:
\begin{itemize}
 \item[(i)] $d_1=d_2$ and there is a line in $\A$ having precisely $d_1$ or $d_1+1$ multiple points,
 \item[(ii)] $d_1 < d_2, \; d_1 +3 \neq d_2$ and there is a line in $\A$ having precisely $d_1+1$ or $d_2$ multiple points,s
 \item[(iii)] $d_1 +3 = d_2$ and there is a line in $\A$ having precisely $d_2$ multiple points,
 \end{itemize}
 then all arrangements lattice isomorphic to $\A$ are also nearly free.
 \end{remark}
 
 Actually, the condition in Proposition \ref{prop:sufficient_cond} ensures us that the nearly free arrangement
 was obtained by the addition of a line with a certain number of multiple points to a free arrangement. In fact, for arrangements with up to $7$ lines, this seems to always be the case.
 
 \begin{prop}
 \label{prop:addition}
 Let $\A$ be a nearly free line arrangement with exponents $d_2 \leq d_2$ and at most $7$ lines. Then $\A$ is obtained by the addition of a line $H$ to a free arrangement such that $H$ intersects the free arrangement in $d_2+1$ points.
 \end{prop}
 
 \begin{proof}
 We make a complete inventory of lattice isomorphism classes of line arrangements $\A$ with $|\A| \leq 7$ (along with their realisation spaces), and we look for the ones that may contain nearly free arrangements. \\
 
 It is fairly easy to see that near freeness is combinatorial for arrangements of up to $7$ lines. Take a nearly free such arrangement $\A$ with exponents $d_1 \leq d_2$. Then $d_1 \in \{2,3\}$ and  obviously the inequality \ref{eq:Cor1.7D2} is satisfied unless $\A$ is generic, Proposition \ref{prop:Cor1.7D2} thus proving our claim. If $\A$ is generic, then, by \cite[Prop. 4.7(4)]{DIM}, near freeness is again combinatorial: if $|\A| = 4$, then the lattice of $\A$ is $L(4,2)$ (in the notations from \cite[Sect.4]{DIM}, $L(d,m)$ is the intersection lattice for a line arrangement having only double points, except a multiple point of multiplicity $m$) and all arrangements in the lattice isomorphism class of $\A$  are nearly free, whereas generic arrangements with cardinal at least $5$ are never nearly free, by the same \cite[Prop. 4.7(4)]{DIM}.\\
 
 Keep in mind that freeness is combinatoric in this range, i.e., once a lattice isomorphism class contains a free arrangement, all other arrangements in the same class are also free.\\
  
 One can only find nearly free arangements for $|\A| \geq 4$, so we look at the intersection lattice types for arrangements $\A, \; 4 \leq |\A| \leq 7$.
 
 For $|\A| = 4$, apart from the generic arrangement, all other arrangements (up to lattice isomorphism) are free; as recalled above, the generic one is nearly free.
 
 For $|\A|  = 5$, all arrangements (up to lattice isomorphism) are either supersolvable, nearly supersolvable (see \cite{DS2} for the definition) or generic. It is well known that arrangements with supersolvable lattices are free (ref). According to \cite[Thm. 4.3]{DS2}, nearly supersolvable arrangements are either free or nearly free, and the choice among the two is also combinatorial.
 
 For $|\A|  = 6$, apart from supersolvable, nearly supersolvable and generic type we have a lattice of type $\tilde L(m_1, m_2)$ (in the notations from \cite[Sect.4]{DIM}, it is the intersection lattice of a line arrangement that consists of two pencils of $m_1 \geq 2$, respectively $m_2 \geq 2$ lines that intersect generically).

A slightly more elaborate approach is necessary for $|\A| = 7$. One computes for an arrangement of each lattice isomorphism type the minimal degree relations and identifies the nearly free (lattice isomorphism classes of) arrangements by the characterisation of near freeness in \cite[Thm 1.3]{D1}. There are, up to lattice isomorphism, $4$ nearly free arrangements: one nearly supersolvable, one of type $L(7,5)$ and the two arrangements in Figure \ref{fig:7lines_arr} below.\\

Finally, up to lattice isomorphism, we have one nearly free arrangement $|\A|$ with $|\A| = 4$, one nearly free arrangement $|\A|$ with $|\A| = 5$, two nearly free arrangements $|\A|$ with $|\A| = 6$ and $4$ nearly free arrangements $|\A|$ with $|\A| = 7$. The exponents $d_1 \leq d_2$ are easily computed. It is easy to see that these arrangements all satisfy our claim,  that is,  all are obtained by the addition of a line to a free (actually, supersolvable) arrangement, a line that intersects the free arrangement in $d_2+1$ points.
 \end{proof}

 \begin{remark}
 \label{rem:7linii}
In the course of the proof of the previous proposition we found two nearly free arrangements with $7$ lines, having the same pattern of multiple points, $n_3 = 5, \; n_2 = 6, n_{>3}=0$, and the same exponents, $(3,4)$, and both can be obtained by the addition of a line $H$ to the same free $6$ lines arrangement, such that the said line intersects the free arrangement in $5$ distinct points. {\it The two arrangements are not lattice isomorphic.} They are pictured in Figure \ref{fig:7lines_arr} (where $H$ is the dashed line). The arrangement $\A_1$ can described by the equations $xyz(x-y)(x+z)(y+z)(x+ay+z)$, where $a$ a complex number $\neq 0,1$, while  the arrangement $\A_2$ can be described by the equations $xyz(x-cy)(x+z)(y+z)(x-cy+(1-c)z)$, where $c$ a complex number $\neq 0,1$.
 \end{remark}

\begin{figure}[h]
\label{fig:7lines_arr}
\begin{tikzpicture}[scale=0.5]

\draw[style=thick] (-12,1) -- (-3,1);
\node at (-10,1){$\bullet$};
\node at (-6.75,1){$\bullet$};

\draw[style=thick, rotate around={36: (-10,1)}] (-12,1) -- (-3,1);
\draw[style=thick, rotate around={56: (-10,1)}] (-12,1) -- (1,1);
\draw[style=thick, rotate around={74: (-6.75,1)}] (-8,1) -- (3,1);
\draw[style=thick, rotate around={90: (-6.75,1)}] (-12.5,1) -- (-1,1);
\draw[style=thick, rotate around={115: (-4.5,1)}] (-6,1) -- (2,1);
\node at (-6.75,5.8){$\bullet$};
\node at (-5.9,4){$\bullet$};
\node at (-4.35,9.4){$\bullet$};
\draw[style=thick, dashed, rotate around={80: (-5.8,1)}] (-11.5,1) -- (4,1);
\node at (-10,-1.5) {($\A_1$)};
\node at (-5.5,0.1) {$H$};

\draw[style=thick] (4,1) -- (13,1);
\node at (6,1){$\bullet$};
\node at (9.25,1){$\bullet$};

\draw[style=thick, rotate around={36: (6,1)}] (4,1) -- (13,1);
\draw[style=thick, rotate around={56: (6,1)}] (4,1) -- (17,1);
\draw[style=thick, rotate around={74: (9.25,1)}] (8,1) -- (19,1);
\draw[style=thick, rotate around={90: (9.25,1)}] (8,1) -- (15,1);
\draw[style=thick, rotate around={115: (11.5,1)}] (10,1) -- (18,1);
\node at (9.25,5.8){$\bullet$};
\node at (10.1,4){$\bullet$};
\node at (9.2,3.3){$\bullet$};
\draw[style=thick, dashed, rotate around={142: (12.2,1)}] (11,1) -- (20,1);
\node at (6,-1.5) {($\A_2$)};
\node at (13.5,0.1) {$H$};
\end{tikzpicture}
\caption{$2$ nearly free arrangements}
\end{figure}
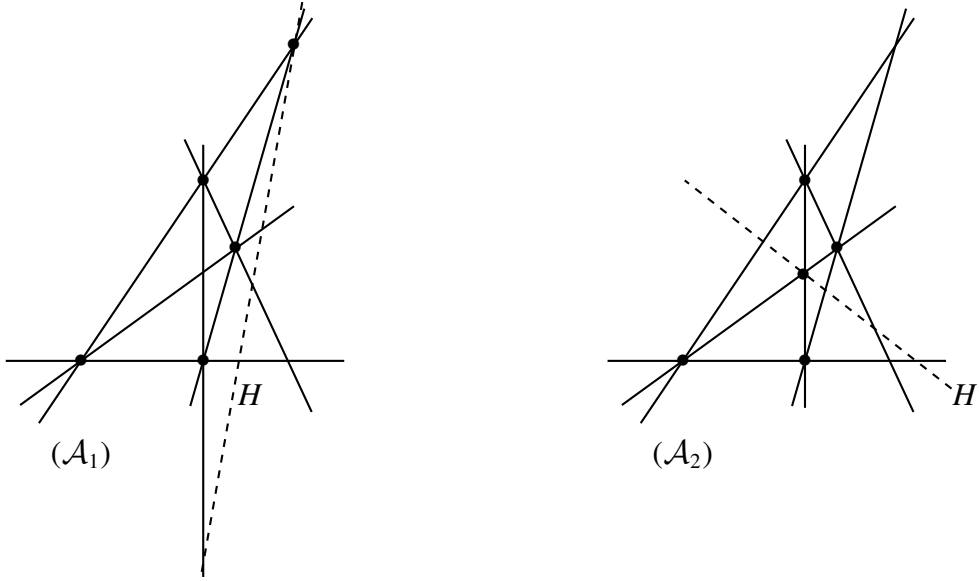

 \begin{prop}
 \label{prop:multiarr}
 Assume $\A, \; |\A|=2d+1$, is a nearly free arrangement with exponents $(d, d+1)$ and there is a line $H$ in $\A$ containing precisely $d$ triple points. Then any other arrangement in the lattice isomorphism class of $\A$ is again nearly free. 
 \end{prop}
 
 \begin{proof}
 Any arrangement $\A'$ in the lattice isomorphism class of $\A$ has characteristic polynomial $\chi(\A', t) = (t-d)^2+1$. Let $H'$ be the line in $\A'$ corresponding to $H$ through the lattice isomorphism and $exp(\A'^{H'}, m^{H'})$  be  the exponents of the Ziegler restriction of $\A'$ to $H'$. 
 By \cite[Prop 1.23(iii)]{Y}, $exp(\A'^H, m^{H'}) = (d,d)$. Apply now \cite[Thm 5.5]{AD} to conclude that $\A'$ must also be nearly free.
 \end{proof}

\begin{prop}
\label{prop:upper_limit_points}
Let $\A$ be a nearly free arrangement with exponents $d_1 \leq d_2$. Then for any $H \in \A, \; n_H \leq d_2+1$.
\end{prop}

\begin{proof}
By \cite[Thm 5.3]{AD}, any Ziegler restriction $(\A^H, m^H))$ has exponents either $(d_1-1, d_2)$ or $(d_1, d_2-1)$. On the other hand, if an $H$ such that $n_H > d_2+1$ would exist, then its Ziegler restriction would have as exponents $(d-n_H, n_H-1)$, by \cite[Prop 1.22(ii)]{Y}, so either $d-n_H = d_1-1$ or $d-n_H = d_1$, that is, $n_H  \in \{d_2, d_2+1\}$, contradiction to the assumption $n_H > d_2+1$.
\end{proof}

Before stating our next result, we need to recall two results, stating sufficient, respectively necessary conditions for the combinatoriality  of near freeness, results that we will extensively use further.

\begin{prop}\cite[Cor. 1.7]{D2}
\label{prop:Cor1.7D2}
Let $\A$ be a nearly free line arrangement with exponents $d_1 \leq d_2$. If 
\begin{equation}
\label{eq:Cor1.7D2}
m(\A) \geq d_1
\end{equation}
then any other line arrangement in the lattice isomorphism class of $\A$ is also nearly free.
\end{prop}

\begin{prop}\cite[Prop. 1.3]{D2}
\label{prop:Prop1.3D2}
Let $\A$ be a line arrangement with $|\A| = d$. If $\A$ is free or nearly free with exponents $d_1 \leq d_2$, then 
\begin{equation}
\label{eq:ineqProp1.3D2}
m(\A) \geq \frac{2d}{d_1+2}
\end{equation}
\end{prop}
 
 \begin{theorem}
 \label{thm:nfree}
 The property of being nearly free is combinatorial for arrangements of up to $12$ lines in the complex projective plane.
 \end{theorem}
 
 \begin{proof}
 Let $\A$ be a nearly free arrangement with exponents $d_1 \leq  d_2$ and $|\A| \leq 12$. We need to see that any other arrangement in its lattice isomorphism class is also nearly  free.
 We can further assume that $m(\A) < d_1$, otherwise the claim of the theorem is proved by \eqref{eq:Cor1.7D2}. 
 
By Proposition \ref{prop:addition}, all nearly free arrangements with $|\A| \leq 7$ satisfy the hypothesis of Proposition \ref{prop:sufficient_cond}, so the claim of the theorem holds.\\
  
For the rest of the cases we no longer know whether or not near free arrangements can be obtained by addition of a line to a free arrangement, however cases $|\A| \in \{8,\dots, 12\}$ admit an uniform approach.

 Since $\A$ cannot be generic (generic arrangements with at least $5$ lines are not nearly free, by \cite[Prop 4.7(4)]{DIM}), we have in any case $m(\A) \geq 3$, so, by the inequality \eqref{eq:Cor1.7D2}, it is enough to consider $d_1 \geq 4$. We can moreover assume that $\A$ does not contain lines with at most $4$ multiple points, since this case is covered by \cite[Thm 1.5]{AD}. Even in the case when the characteristic polynomial is a perfect square and it does not distinguish a priori between free and nearly free arrangements, we are in range $|\A| \leq 13$ and one cannot have a free arrangement within the same lattice isomorphism class as the nearly free $\A$.

By \eqref{eq:Cor1.7D2} and  \eqref{eq:ineqProp1.3D2}, for $|\A| = 8$, we get $m(\A) =3, \;(d_1, d_2) = (4,4)$, for $|\A| = 9$, we get $m(\A) =3, \;(d_1, d_2) = (4,5)$ and for $|\A| = 10$, we get $m(\A) \in \{3,4\}, \;(d_1, d_2) = (5,5)$.  We use Proposition  \ref{prop:sufficient_cond}, Remark \ref{rem:sufficient_cond} and Proposition \ref{prop:upper_limit_points} (and additionally Proposition \ref{prop:multiarr}, for $|A|$ odd) to restrict the cases of the nearly free arrangements we have to consider and to add further constraints  on those  cases (for instance, to restrict the possible types of lines in $\A$).

For $|\A| = 8$, for instance, one can only have lines with at most $d_2+1=5$ multiple points (Proposition \ref{prop:upper_limit_points}), hence, by the previous discussion, it is enough to prove the theorem in the case when $\A$ has only lines with precisely $5$ multiple points. Then we readily apply Remark \ref{rem:sufficient_cond} (i), and we are done.

For $|\A| = 9$, one can only have lines $H \in \A$ with at most $d_2+1=6$ multiple points, so we are left with two types of lines for $\A$:

\begin{itemize}
\item[(1)] $n^H_3 = 3, \; n^H_2 = 2$
\item[(2)] $n^H_3 = 2, \; n^H_2 = 4$
\end{itemize}

So, $\A$ has either a line with $5$ multiple points, and we may apply Remark \ref{rem:sufficient_cond} (ii), or a line with $6$ multiple points, and we apply Proposition  \ref{prop:sufficient_cond}.

For $|\A| = 10$, one can only have lines with at most $d_2+1=6$ multiple points, hence there are only lines with $5$ or $6$ multiple points (case already solved by Remark \ref{rem:sufficient_cond} (i)). \\

Assume $|\A| =11$.  Then $d_1 \in \{4,5\}$. But $d_1 = 4$ implies $m(\A) = 3$ and this contradicts the inequality \eqref{eq:ineqProp1.3D2}. We are left with $(d_1,d_2) = (5,6)$. By \eqref{eq:Cor1.7D2} and \eqref{eq:ineqProp1.3D2}, $m(\A) = 4$.

We may further assume that $\A$ does not have lines with $5$ triple points, since this case is already covered by Proposition \ref{prop:multiarr}.
Denote by $d,e$ the number of lines $H \in \A$ of each of the following types:
\begin{itemize}
\item[(d)] $n^H_4 = 2, \;   n^H_3 = 1, \; n^H_2 = 2$ 
\item[(e)] $n^H_4 = 1, \;   n^H_3 = 3, \; n^H_2 = 1$
\end{itemize}

As in the proof of the previous theorem, we can encode the combinatorial information into a linear system of (non-homogeneous) equations, taking also into account near freeness (Tjurina number almost maximal, \cite[Thm 1.3]{D1}). Denote by $a,b,c$ the number of double, triple, respectively quadruple, points of the arrangement $\A$. The first and second equations in the system are derived from equalities in \cite[Thm 1.3]{D1}, respectively \eqref{eq:suma_comb}. 

Notice that, apart from possible lines with $6$ or $7$ multiple points, these are all the types of possible lines in $\A$. No lines with more than $7$ multiple points are allowed, by Proposition  \ref{prop:upper_limit_points}.

Also, if $\A$ would have a line with either $6$ or $7$ multiple points, then any arrangement in its lattice isomorphism class would  also be nearly free, by  Proposition  \ref{prop:sufficient_cond} and Remark \ref{rem:sufficient_cond} (ii).

So we are left with proving the claim of the theorem in the situation when $\A$ does not contain lines with $6$ or $7$ multiple points. Then $\A$ only contains lines of type $(d),(e)$ and the next system should admit a solution. \\

\noindent$
| \; a+4b+9c=74\\
| \; a+3b+6c=55\\
| \; -2a+2d+e=0\\
| \; -3b+d+3e=0\\
| \; -4c+2d+e=0\\
| \; d+e=11
$
\\

We call a solution admissible if its  components are non-negative integers. A Normaliz (\cite{Nmz}) aided computation (we compute the integral points in a polytope) shows that there are no admissible solutions to the system. \\

We treat the same the case  $|\A|=12$. By \eqref{eq:ineqProp1.3D2} and \eqref{eq:Cor1.7D2}, there are two possible sets of exponents to consider.

{\it Case I} The arrangement has exponents $(5,7)$ and $m(\A) = 4$.\\
From Proposition \ref{prop:upper_limit_points} we get that all lines $H \in \A$ have $n_H \leq 8$. Moreover, by Remark  \ref{rem:sufficient_cond} (ii), we may consider only the situation when $n_H \neq 6,7$ and by Proposition we may add the additional restriction $n_H \neq 8$. We are left with proving the theorem for $\A$ line arrangement in which all lines have precisely $5$ multiple points. The possible types of lines $H \in \A$ are:
\begin{itemize}
\item[(d)]  $n^H_4 = 3, \;   n^H_3 = 0, \; n^H_2 = 2$
\item[(e)] $n^H_4 = 2, \;  n^H_3 = 2, \; n^H_2 = 1$
\item[(f)] $n^H_4 = 1, \;  n^H_3 = 4, \; n^H_2 = 0$
\end{itemize}

The corresponding system is (where $*$ is the number of lines of type $(*)$ and $a,b,c$ denotes the number of double, triple, respectively quadruple, points of the arrangement):\\

\noindent$
| \; a+4b+9c=90\\
| \; a+3b+6c=66\\
| \; -2a+2d+e=0\\
| \; -3b+2e+4f=0\\
| \; -4c+3d+2e+f=0\\
| \; d+e+f=12
$
\\

{\it Case II} The arrangement has exponents $(6,6)$ and $m(\A) \in \{3,4,5\}$.\\
From Proposition \ref{prop:upper_limit_points} we get that all lines $H \in \A$ have $n_H \leq 7$. Moreover, by Remark  \ref{rem:sufficient_cond} (i), we may consider only the situation when $n_H \neq 6,7$, hence are left with proving the theorem for $\A$ line arrangement in which all lines have precisely $5$ multiple points. The possible types of lines $H$ are:
\begin{itemize}
\item[(d)] $n^H_5 = 2, \;n^H_4 = 0,\;  n^H_3 = 0, \;n^H_2 = 3$	
\item[(e)]	 $n^H_5 = 1, \;n^H_4 = 1, \;  n^H_3 = 1, \; n^H_2 = 2$	
\item[(f)]	 $n^H_5 = 1, \;n^H_4 = 0,\;  n^H_3 = 3, \; n^H_2 = 1$
\item[(g)]	 $n^H_5 = 0, \;n^H_4 = 3, \;  n^H_3 = 0, \; n^H_2 = 2$
\item[(h)]	 $n^H_5 = 0, \;n^H_4 = 2, \;  n^H_3 = 2, \; n^H_2 = 1$
\item[(i)]	 $n^H_5 = 0, \;n^H_4 = 1, \;  n^H_3 = 4, \; n^H_2 = 0$
\end{itemize}

The corresponding system is (where $*$ is the number of lines of type $(*)$ and $a,b,c,0$ denotes the number of double, triple, quadruple, respectively quintuple points of the arrangement):\\

\noindent$
| \; a+4b+9c+16o=90\\
| \; a+3b+6c+10o=66\\
| \; -2a+3d+2e+f+2g+h=0\\
| \; -3b+e+3f+2h+4i=0\\
| \; -4c+e+3g+2h+i=0\\
| \; -5o+2d+e+f=0\\
| \; d+e+f+g+h+i=12
$\\

In both cases a Normaliz aided computation for the above systems of linear non-homogeneous equations shows that there are no integer non-negative solutions, so no nearly free arrangements with the constraints described by the above two cases exist.

\end{proof}

 For $|\A| = 13$ it is not that straightforward to show that near freeness is combinatorial. The resulting systems do have admissible solutions, even though these solutions may not be realizable as combinatorics of a complex projective line arrangement. However, we are able to reduce the Terao conjecture for $14$ lines arrangements to a near freeness problem concerning $13$ lines arrangements. 

\begin{prop}
\label{prop:14free}
The Terao conjecture is true for arrangements of $14$ lines if near freeness is combinatorial for $13$ lines arrangements.
\end{prop}

\begin{proof}
Let $\A$ be a free $14$ lines arrangement. By \cite[Cor. 5.5]{A}, we only have to consider the case when the exponents of $\A$ are $(6,7)$. Apply \cite[Cor.1.4]{D2} and \eqref{eq:ineqProp1.3D2} to deduce $m(\A) \in \{4,5\}$. Since the characteristic polynomial of $\A$ is not a perfect square, the theorem is proved in the case when $\A$ has a line with at most $4$ multiple points, by \cite[Thm. 1.5]{AD}; so we can further assume that $\A$ only contains lines with at least $5$ multiple points. By \cite[Thm. 2.7]{ACKN}, if there exists $H \in \A$ such that $n_H$, the number of multiple points on $H$, is at least $7$, the Terao conjecture is true for $\A$.

Hence we only need to prove the claim for arrangements $\A$ of $14$ lines with exponents $(6,7)$ and $m(\A) \in \{4,5\}, \;n_H \in \{5,6\}$, for all $H \in \A$. Notice that if $\A$ contains a line $H$ with $n_H = 6$, then, by Theorem \ref{thm:5.11_del_restr}, since the property of being nearly free for arrangements of $13$ lines is combinatorial (Theorem \ref{thm:main}), and the exponents of two lattice isomorphic nearly free arrangements coincide (see \cite[Cor. 3.5]{DS}), our claim is proven. So, we only need to show that $\A$ contains a line with precisely $6$ multiple points. Assume next that $\A$ does not contain any such line. Then the lines of $\A$ must be of one of the following types:

\begin{enumerate}
\item[(d)] $n^H_5 = 2, \;n^H_4 = 0,\;  n^H_3 = 2, \;n^H_2 = 1$
\item[(e)] $n^H_5 = 2, \;n^H_4 = 1,\;  n^H_3 = 0, \;n^H_2 = 2$
\item[(f)] $n^H_5 = 1, \;n^H_4 = 2,\;  n^H_3 = 1, \;n^H_2 = 1$
\item[(g)] $n^H_5 = 1, \;n^H_4 = 1,\;  n^H_3 = 3, \;n^H_2 = 0$
\item[(h)] $n^H_5 = 0, \;n^H_4 = 4,\;  n^H_3 = 0, \;n^H_2 = 1$
\item[(i)] $n^H_5 = 0, \;n^H_4 = 3,\;  n^H_3 = 2, \;n^H_2 = 0$
\end{enumerate}

The combinatorial data of $\A$ is encoded in the following system:\\

\noindent $
| \; a+4b+9c+16o=127\\
| \; a+3b+6c+10o=91\\
| \; -2a+d+2e+f+h=0\\
| \; -3b+2d+f+3g+2i=0\\
| \; -4c+e+2f+g+4h+3i=0\\
| \; -5c+2d+2e+f+g=0\\
| \; d+e+f+g+h+i=14
$
\\

\noindent where $a,b,c,o$ are the number of double, triple, quadruple, respectively quintuple points of the arrangement and $d,e,f,g,h,i$ encode the number of lines of the corresponding types. We add as constraints that all variables should be positive integers. Again a Normaliz computation shows that there is no admissible solution. It follows that our assumption that $\A$ does not contain a line with $6$ multiple points is false, so the claim of the theorem is proven.
\end{proof}


\begin{thebibliography}{00}

\bibitem{A} T.~Abe,
{\em Roots of characteristic polynomials and intersection 
points of line arrangements},
J. of Singularities 8 (2014),100--117.

\bibitem{A2} T.~Abe,
{\em Exponents of 2-multiarrangements and
freeness of 3-arrangements},
\href{arxiv.org/pdf/1005.5276.pdf}{arXiv:1005.5276}

\bibitem{ACKN} T.~Abe, M.~ Cuntz, H.~ Kawanoue,  T.~Nozawa
{\em  Non-recursive freeness and non-rigidity},
Discrete Math. 339 (2016), no. 5, 1430--1449

\bibitem{AD} T.~Abe, A.~Dimca,
{\em On the splitting types of bundles of logarithmic vector fields along plane curves}, 
\href{https://arxiv.org/pdf/1706.05146.pdf}{arXiv:1706.05146}

\bibitem{D1} A.~Dimca, 
{\em Freeness versus maximal global Tjurina number for plane curves}, 
Math. Proc. Cambridge Phil. Soc.  163 (2017), 161--172.

\bibitem{D2} A.~Dimca,
{\em Curve arrangements, pencils, and Jacobian syzygies},
Michigan Math. J. 66 (2017), 347--365.

\bibitem{DHA}  A. Dimca,   {\em Hyperplane Arrangements: An Introduction}, Universitext, Springer, 2017.

\bibitem{DIM} A.~Dimca, D.~Ibadula, A.~Macinic,
{\em Numerical invariants and moduli spaces for line arrangements}
\href{https://arxiv.org/abs/1609.06551}{arXiv:1609.06551}

\bibitem{DS} A.~ Dimca, G.~ Sticlaru
{\em Free and nearly free curves vs. rational cuspidal plane curves}
Publ. RIMS Kyoto Univ. 54 (2018), 163--179.

\bibitem{DS2} A.~ Dimca, G.~ Sticlaru
{\em On supersolvable and nearly supersolvable line arrangements}
 \href{https://arxiv.org/abs/1712.03885}{arXiv:1712.03885}
 
\bibitem{FV}  D. Faenzi, J. Vall\`es, 
{\em Logarithmic bundles and line arrangements, an approach via the standard construction}, J. London.Math.Soc.
{90} (2014), 675--694.

\bibitem{H} F.~Hirzebruch,
{\em Arrangements of lines and algebraic surfaces}, Arithmetic and geometry, Vol. II, Progr.
Math. 36, Birkhauser, Boston, Mass., 1983, 113--140.

\bibitem{Nmz} W. Bruns, B. Ichim, T. R\"{o} mer, R. Sieg and C. S\"{o} ger: Normaliz. 
{\em Algorithms for rational cones and affine monoids}
\href{https://www.normaliz.uni-osnabrueck.de}{https://www.normaliz.uni-osnabrueck.de}

\bibitem{OT} P.~Orlik, H.~Terao,
{\em Arrangements of hyperplanes}, Grundlehren Math. Wiss., vol.~300,
Springer-Verlag, Berlin, 1992.

\bibitem{Y0} M. ~Yoshinaga, 
{\em On the freeness of 3-arrangements}
 Bull. London Math. Soc. 37 (2005),  126--134.

\bibitem{Y} M. ~Yoshinaga,
{\em Freeness of hyperplane arrangements and related topics}, Annales de la Facult\'e des Sciences de Toulouse Vol XXIII,
 (2014),  483--512.

\end{thebibliography}
\end{document}